\documentclass[10pt,a4paper,reqno]{amsart}
\usepackage{a4wide}
\usepackage{amssymb, amsmath, amsthm,amsfonts}
\usepackage[utf8]{inputenc}
\usepackage{tikz}
\usetikzlibrary{matrix,arrows,positioning,backgrounds,shapes}
\usepackage{multirow}

\theoremstyle{plain} 
\newtheorem*{lemma}{Lemma}
\newtheorem*{theo*}{Theorem}

\newcommand{\PSL}{\operatorname{PSL}}

\newcommand{\Ker}{\operatorname{Ker}}
\newcommand{\Z}{\mathrm{Z}}
\newcommand{\ZZ}{\mathbb{Z}}
\newcommand{\QQ}{\mathbb{Q}}
\newcommand{\FF}{\mathbb{F}}
\newcommand{\V}{\mathrm{V}}
\newlength{\heightofhw}
\newcommand{\eigbox}[2]{\ensuremath{#1 \times \boxed{\rule{0cm}{\heightofhw} #2}}}

\definecolor{LinkColor}{rgb}{0,0,0} 
\usepackage[colorlinks=true,linkcolor=LinkColor,citecolor=LinkColor,urlcolor= LinkColor, naturalnames, hyperindex, pdfstartview=FitH, bookmarksnumbered, plainpages]{hyperref}
    
\begin{document}

\title[Torsion subgroups in the units of the integral group ring of $\PSL(2,p^3)$]{Torsion subgroups in the units of the\\ integral group ring of $\PSL(2,p^3)$}
\author{Andreas B\"achle}
\address{Vakgroep Wiskunde, Vrije Universiteit Brussel, Pleinlaan 2, 1050 Brussels, Belgium}
\email{\href{mailto:abachle@vub.ac.be}{abachle@vub.ac.be}}
\author{Leo Margolis}
\address{Fachbereich Mathematik, Universit\"{a}t Stuttgart, Pfaffenwaldring 57, 70569 Stuttgart, Germany}
\email{\href{mailto:leo.margolis@mathematik.uni-stuttgart.de}{leo.margolis@mathematik.uni-stuttgart.de}}
\date{\today}
\thanks{The first author is supported by the Research Foundation Flanders (FWO - Vlaanderen).}

\begin{abstract}
 We show that for every prime $r$ all $r$-subgroups in the normalized units of the integral group ring of $\PSL(2,p^3)$ are isomorphic to subgroups of $\PSL(2,p^3)$. This answers a question of M.~Hert\-weck, C.R.~H{\"o}fert and W.~Kimmerle for this series of groups.
\end{abstract}

\maketitle

 Let $G$ be a finite group and $\ZZ G$ its integral group ring. A natural question is how far the structure of $G$ is reflected in the structure of the torsion part of $\V(\ZZ G)$, the group of normalized units of the integral group ring of $G$, and vice versa. One of the first questions studied in this context is, whether a finite subgroup of $\V(\ZZ G)$ is necessarily isomorphic to a subgroup of $G$, which was already raised in G.~Higman's PhD thesis \cite[Section 5]{Hig} (cf.\ also the article \cite{Sand1} outlining this thesis). It was later explicitly stated as a problem in \cite[Problem 5.4]{Sand2}.
 
 A first step towards an answer to the above question lies in the consideration of $p$-subgroups.  It is known that if $G$ has cyclic Sylow $p$-subgroups then all $p$-subgroups of $\V(\ZZ G)$ are cyclic; for $p = 2$  see \cite[Proposition]{C2C2} (the proof involves the Brauer-Suzuki theorem) and cf.\ \cite[Corollary 1]{CpCp} for odd primes $p$. If the Sylow $2$-subgroups of $G$ are elementary abelian, then so are all $2$-subgroups of $\V(\ZZ G)$ and in particular they are isomorphic to subgroups of $G$ (combining  \cite[Corollary 4.1]{CL} and \cite[Corollary 1.7]{Saksonov}, cf.\ also \cite[Lemma (37.3)]{SehgalBook2}). But it is not even known in general whether abelian Sylow $p$-subgroups of $G$ cause that the $p$-subgroups of $\V(\ZZ G)$ are abelian. However, this assertion holds for solvable $G$ \cite[Proposition 2.11]{DoJu}.
 
 In the concluding remarks of \cite{HHK} the question whether the $p$-subgroups of $\V(\ZZ \PSL(2,p^f))$ are (elementary) abelian is highlighted. For $p = 2$ or $f \leq 2$ there is an even stronger assertion, recently proved by the second author \cite{Mar_Sylow}: for all primes $r$ the $r$-subgroups of $\V(\ZZ \PSL(2,p^f))$ are conjugate by a unit of the corresponding rational group algebra to a subgroup of the group base $\PSL(2,p^f)$. It is also known that any subgroup of $\V(\ZZ \PSL(2,p^f))$ of the same order as $\PSL(2,p^f)$ is conjugate in the rational group algebra to the group base $\PSL(2,p^f)$ \cite[Propositions 3.2 and 4.1]{Bleher}.
 
In this article we prove that for all primes $p$ the $p$-subgroups of the group $\V(\ZZ \PSL(2,p^3))$ are abelian. We obtain the following result:
 \begin{theo*} Let $p$ be a prime and set $G = \PSL(2,p^3)$. Then for all primes $r$, the $r$-subgroups of $\V(\ZZ G)$ are isomorphic to subgroups of $G$. \end{theo*}

 The method of choice here is what one could call a ``non-cyclic HeLP-method''. For a finite group $G$ consider a finite subgroup $U$ of $\V(\ZZ G)$. Every ordinary representation of $G$ can be linearly extended to a representation of $\ZZ G$ and then restricted to a representation $D$ of $\V(\ZZ G)$. Let its character be denoted by $\chi$. If $\psi$ is a character of $U$, then for the inner product we necessarily have \[ \left< \chi|_U, \psi \right>_U \in \ZZ_{\geq 0}. \] For cyclic $U$ this can be expressed in a explicit formula, whose application is known as the HeLP-method, cf.\ \cite[Theorem 1]{LutharPassiA5} and \cite[Section 4]{HertweckBrauer}. The method for non-cyclic $U$ was used for the first time in the PhD-thesis of C.R. H{\"o}fert \cite[page 58]{HoefPhD} and later in \cite{CpCp} and \cite[page 2]{HHK}.
 
 We will introduce some notation. Let $u = \sum_{g \in G} z_g g \in \ZZ G$ be a normalized torsion unit. For $x^G$, the conjugacy class of the element $x \in G$ in $G$, we denote by \[ \varepsilon_x(u) = \sum_{g \in x^G} z_g \] the \emph{partial augmentation of $u$ at $x$}. The partial augmentations will provide us with restrictions on the possible eigenvalues of $D(u)$ and vice versa. By the so-called Berman-Higman Theorem $\varepsilon_1(u) = 0$ if $u \not= 1$ \cite[Theorem 10]{Hig}, \cite[Lemma 2]{Ber} (cf.\ also \cite[Proposition (1.4)]{SehgalBook2}). Also $\varepsilon_x(u) = 0$ if the order of $x$ does not divide the order of $u$ \cite[Proposition 2.2]{HertweckBrauer}. Moreover, the order of a finite subgroup of $\V(\ZZ G)$ divides the order of $G$ \cite[Corollary 1.7]{Saksonov}, cf. also \cite[Lemma (37.3)]{SehgalBook2}, and its exponent divides the exponent of $G$ \cite[Corollary 4.1]{CL}.  We will use these facts in the sequel without further mention.
 
 Let $A$ be a complex matrix of finite order. Assume that $A$ has eigenvalues $\alpha_1, ..., \alpha_j$ each with multiplicity $m_1$, eigenvalues $\beta_1, ..., \beta_k$ each with multiplicity $m_2$ and eigenvalues $\gamma_1, ..., \gamma_\ell$ each with multiplicity $m_3$, then we indicate this by 
\[ A \sim \left(\eigbox{m_1}{\alpha_1, ..., \alpha_j}, \eigbox{m_2}{\beta_1, ..., \beta_k}, \eigbox{m_3}{\gamma_1, ..., \gamma_\ell} \right). \]

\section*{Proof of the theorem.}

 Set $G = \PSL(2,p^3)$. By a result of Dickson \cite[260.]{Dickson} (see also \cite[8.27 Hauptsatz]{Huppert1}), the Sylow $r$-subgroups of the simple groups $\PSL(2,p^f)$ are elementary abelian for $r = p$, cyclic for odd $r \not= p$, and dihedral groups if $r = 2 \not= p$. By the results cited above we obtain that the $r$-subgroups of $\V(\ZZ G)$ are isomorphic to subgroups of $G$, provided the Sylow $r$-subgroups are cyclic. The case of elementary abelian Sylow $2$-subgroups is also handled by the remarks in the introduction. If the Sylow $2$-subgroups are dihedral, the result is obtained in \cite[Theorem 2.1]{HHK}. It remains the case $r = p \geq 3$. Note that in this case the Sylow $p$-subgroups of $G$ are elementary abelian of order $p^3$.
 
 Let $H$ be a finite $p$-subgroup of $\V(\ZZ G)$. Hence $|H| \leq p^3$ and $\exp H \mid p$. Assume that $H$ is not isomorphic to a subgroup of $G$, then, by the classification of all $p$-groups up to order $p^3$, it is a so-called Heisenberg group. Thus there are elements $z, b, c \in H$ such that \begin{align} H = \langle z,\ b,\ c\ |\ z^p = b^p = c^p = 1,\ z \in \Z(H),\ c^{-1}bc = zb \rangle \simeq (C_p \times C_p) \rtimes C_p.\label{eq:defH}\end{align}
 
 We will use the non-cyclic HeLP-method to show that $H$ does not exist. In $G$ there are exactly two conjugacy classes of elements of order $p$, let $g$ and $h$ be representatives of these classes. In Table \ref{CT_PSL_2_p3} we list two irreducible characters \cite[Theorem 38.1]{Dornhoff}, one of them we will use in the remainder of the proof. Let $\epsilon \in \{1, -1\}$ such that $p \equiv \epsilon \bmod 4$.
 
  \begin{table}[ht]\centering
 \begin{tabular}{cccc}  \hline
 & $1$ & $g$ & $h$  \\ \hline  \hline \\[-.38cm]
 $\eta$ & $\frac{p^3+\epsilon}{2}$ & $\frac{\epsilon + \sqrt{\epsilon p^3}}{2}$ & $\frac{\epsilon - \sqrt{\epsilon p^3}}{2}$ \\
 $\eta'$ & $\frac{p^3+\epsilon}{2}$ & $\frac{\epsilon - \sqrt{\epsilon p^3}}{2}$ & $\frac{\epsilon + \sqrt{\epsilon p^3}}{2}$ \\[0.05cm] \hline
 \end{tabular} \\[.1cm]
 \caption{Part of the character table of $G = \PSL(2,p^3)$, $p \geq 3$, $\epsilon \in \{1, -1\}$ such that $p \equiv \epsilon \bmod 4$.}\label{CT_PSL_2_p3}
 \end{table}
 
 Let $\zeta = \exp\left(2\pi i / p\right)\in \mathbb{C}$, a primitive $p$-th root of unity, $Q$ be a set of integral representatives of the quadratic residues in $(\ZZ / p \ZZ )^\times $ and $N$ be a set of integral representatives of the non-quadratic residues in $(\ZZ / p \ZZ)^\times$. We will also use the Gaussian sums cf.\ \cite{Gauss} (also \cite[Proof of (8.6) in Chapter I]{Neukirch_english})
 
 \begin{equation} \begin{split}
    \sqrt{\epsilon p} & = 1 + 2\sum_{q \in Q} \zeta^q, \\
   -\sqrt{\epsilon p} & = 1 + 2\sum_{n \in N} \zeta^n.
 \label{eq:gs}\end{split} \end{equation}
 
 Let $D$ be a representation affording $\eta$. Let $u \in H \setminus \{1\}$ and $\alpha \in \ZZ_{\geq 0}$ such that $(\varepsilon_g(u) , \varepsilon_h(u)) = (\alpha + 1, - \alpha)$ or $(\varepsilon_g(u) , \varepsilon_h(u)) = (-\alpha, \alpha + 1)$.  By slight abuse of notation we denote from here on by $\eta$ also the restriction $\eta|_H$. For the first possibility of the partial augmentations of $u$ we obtain, using the first equation of \eqref{eq:gs},
  \begin{align} \eta(u) &= (\alpha + 1) \left( \frac{\epsilon + \sqrt{\epsilon p^3}}{2} \right) + (-\alpha) \left(\frac{\epsilon - \sqrt{\epsilon p^3}}{2} \right) \nonumber \\ &= \frac{1}{2} \left(\epsilon + (2\alpha + 1) p\sqrt{\epsilon p} \right) \nonumber \\ &= \frac{p + \epsilon}{2} + \alpha p + (2\alpha + 1) p \sum_{q \in Q} \zeta^q. \label{eq:chraval1}\end{align}
  
 For the second possibility for the partial augmentations of $u$ we get, using the second equation of \eqref{eq:gs},
   \begin{align} \eta(u) = \frac{p + \epsilon}{2} + \alpha p + (2\alpha + 1) p \sum_{n \in N} \zeta^n. \label{eq:chraval2}\end{align}
  
  From the character value we get $\frac{p + \epsilon}{2} + \alpha p + \left(\frac{p-1}{2}\right)(2\alpha + 1) p$ eigenvalues of $D(u)$. The other eigenvalues have to sum up to $0$, hence there are
  \begin{align*}
    &\ \frac{1}{p}\left(\frac{p^3+\epsilon}{2} - \left(\frac{p + \epsilon}{2} + \alpha p + \left(\frac{p-1}{2}\right)(2\alpha + 1) p \right) \right) \\ =&\ \frac{1}{2p} \left(p^3 + \epsilon - p - \epsilon - 2\alpha p - (p - 1) ( 2\alpha + 1 )p  \right) \\ =&\ \frac{p^2 - p}{2} - \alpha p
  \end{align*}
  blocks having eigenvalues $1, \zeta, \zeta^2, ..., \zeta^{p-1}$. Thus in the case of $\varepsilon_g(u) > 0$ we have
  \[  D(u) \sim \left(\eigbox{\left(\frac{p + \epsilon}{2}  + \alpha p\right)}{1}, \eigbox{(2\alpha + 1)p}{\zeta^{q_1}, ..., \zeta^{q_{\frac{p-1}{2}}}}, \eigbox{\left(\frac{p^2 - p}{2} - \alpha p \right)}{1, \zeta, ..., \zeta^{p-1}} \right) \] where $Q = \left\{q_1, ...,  q_{\frac{p-1}{2}}\right\}$. In the other case, $\varepsilon_g(u) \leq 0$, we get  \[  D(u) \sim \left(\eigbox{\left(\frac{p + \epsilon}{2}  + \alpha p\right)}{1}, \eigbox{(2\alpha + 1)p}{\zeta^{n_1}, ..., \zeta^{n_{\frac{p-1}{2}}}}, \eigbox{\left(\frac{p^2 - p}{2} - \alpha p \right)}{1, \zeta, ..., \zeta^{p-1}} \right) \] with $N = \left\{ n_1, ..., n_{\frac{p-1}{2}} \right\}$. 
  
  The group $H$ is an extra-special $p$-group, its character theory is well-known, see e.g.\ \cite[Theorem 31.5]{Dornhoff}. $H$ has exactly $p-1$ non-linear irreducible characters which are all of degree $p$. They all vanish on the non-central elements of $H$ and take the values $p\zeta^j$ for $1 \leq j \leq p-1$ on the non-trivial central elements. Moreover $H$ posses $p^2$ linear characters corresponding to the quotient $H/\Z(H) = H/H' \simeq C_p \times C_p$. In particular, they have value $1$ on all central elements of $H$. See table \ref{CT_CpxCp:Cp}.
   \begin{table}[h!t]\centering
 \begin{tabular}{ccccccccc} \hline
 & $1$ & $b$ & $b^2$ & $...$ & $b^{p-1}c^{p-1}$ & $z$ & $...$ & $z^{p-1}$ \\ \hline \hline \\[-2.2ex]
\cline{2-6}
 $\chi_1$ & \multicolumn{5}{|c|}{\multirow{5}{*}{\vspace{-.2cm}{\Large CT($C_p \times C_p$)}}} & $1$ & $...$ & $1$ \\
 $\chi_2$ & \multicolumn{5}{|c|}{}  & $1$ & $...$ & $1$ \\
 $\vdots$ & \multicolumn{5}{|c|}{} & $\vdots$ & & $\vdots$ \\
  $\chi_{p^2 - 1}$ & \multicolumn{5}{|c|}{} & $1$ & $...$ & $1$ \\
 $\chi_{p^2}$ & \multicolumn{5}{|c|}{} & $1$ & $...$ & $1$ \\ \cline{2-6}
 $\psi_1$ &$p$ & $0$ & $0$ & $...$ & $0$ & $p\zeta$ & $...$ & $p\zeta^{p-1}$ \\
 $\vdots$ & $\vdots$ & & & & $\vdots$ & $\vdots$ & & $\vdots$ \\
 $\psi_{p-1}$ &$p$ & $0$ & $0$ & $...$ & $0$ & $p\zeta^{p-1}$ & $...$ & $p\zeta$ \\ \hline
 \end{tabular} \\[.2cm]
  \caption{Character table of $H \simeq (C_p \times C_p) \rtimes C_p$, the box indicates the character table of the quotient group $H/\Z(H) \simeq C_p \times C_p$.}\label{CT_CpxCp:Cp}
 \end{table}
 Now we decompose $\eta = \eta|_H$ into the irreducible characters of $H$.
 
 Let $v \in \Z(H)\setminus \{1 \}$. To obtain the multiplicity of the eigenvalue $1$ of $D(v)$ as calculated above, we must sum up exactly \[ \frac{p + \epsilon}{2} + \alpha p + \frac{p^2 - p}{2} - \alpha p =  \frac{p^2 + \epsilon}{2} \] linear characters of $H$ and hence \[ \frac{1}{p} \left( \frac{p^3 + \epsilon}{2} - \frac{p^2 + \epsilon}{2} \right) = \frac{p^2 - p}{2} \] irreducible non-linear character of $H$.
 
 Now let $w \in H$ be a non-central element. Since every irreducible non-linear character of $H$ vanishes on $w$, the character value $\eta(w)$ is the sum of exactly $\frac{p^2 + \epsilon}{2}$ roots of unity. Thus from \eqref{eq:chraval1} and \eqref{eq:chraval2} we obtain that $(\varepsilon_g(w), \varepsilon_h(w)) = (1, 0)$ or $(\varepsilon_g(w), \varepsilon_h(w)) = (0, 1)$. Furthermore, from the eigenvalues of $D(u)$ calculated in these equations, we get \begin{align}(\varepsilon_g(w), \varepsilon_h(w)) = (1, 0) \qquad \Leftrightarrow \qquad  (\varepsilon_g(w^n), \varepsilon_h(w^n)) = (0, 1) \quad \forall n \in N. \label{eq:power_n}\end{align}
 
 Now we compute the inner product $\langle \eta , \chi \rangle_H$ of $\eta$ with a non-trivial linear character $\chi$ of $H$. We split up the computation of the contributions of different parts of $H$ and omit the global factor $1/|H| = 1/p^3$ until adding all contributions. The contribution of the identity element is \[ \eta(1)\chi(1) = \frac{p^3 + \epsilon}{2}. \] There are $p-1$ more elements in the center of $H$. On $\eta$ half of them takes the value $\frac{p + \epsilon}{2} + \alpha p + (2\alpha + 1) p \sum\limits_{q \in Q} \zeta^q$, the other half the value $\frac{p + \epsilon}{2} + \alpha p + (2\alpha + 1) p \sum\limits_{n \in N} \zeta^n$. Since they all lie in the kernel of $\chi$ the contribution to the inner product is \begin{align*} &\ \frac{p-1}{2} \left(\frac{p + \epsilon}{2} + \alpha p + (2\alpha + 1) p \sum\limits_{q \in Q} \zeta^q \right)  +  \frac{p-1}{2} \left(  \frac{p + \epsilon}{2} + \alpha p + (2\alpha + 1) p \sum\limits_{n \in N} \zeta^n \right) \\ =&\  \frac{p-1}{2} \Big( p + \epsilon + 2\alpha p - (2\alpha + 1) p \Big) \\  =& \    \epsilon\cdot \frac{p - 1}{2}.                                                                                                                                                                                                                                                                                                                                                                                                                                                                                                                                                                                                                                        \end{align*}
 The kernel of $\chi$ contains another conjugacy class of cyclic subgroups consisting of exactly $p$ subgroups. In every such cyclic subgroup $\frac{p-1}{2}$ of the elements take the value $\frac{p+\epsilon}{2} + p\sum\limits_{q \in Q}\zeta^q$ on $\eta$ and $\frac{p-1}{2}$ of the elements take the value $\frac{p+\epsilon}{2} + p\sum\limits_{n \in N}\zeta^n$ on $\eta$. The same computation as above gives the contribution \[p \epsilon\cdot \frac{p - 1}{2}.\]
 
 To calculate the contributions of the other elements we need the following formulas which are direct consequences of the Gaussian sums in \eqref{eq:gs}: \begin{align*} \left(\sum\limits_{q \in Q}\zeta^q\right)\left(\sum\limits_{i \in Q}\zeta^i\right) + \left(\sum\limits_{n \in N}\zeta^n\right)\left(\sum\limits_{j \in N}\zeta^j\right) & = \frac{\epsilon\cdot p + 1}{2} \\ \text{and} \qquad  \left(\sum\limits_{n \in N}\zeta^n\right)\left(\sum\limits_{i \in Q}\zeta^i\right) + \left(\sum\limits_{q \in Q}\zeta^q\right)\left(\sum\limits_{j \in N}\zeta^j\right) & = \frac{-\epsilon\cdot p + 1}{2}. \end{align*}
 
 Every conjugacy class of non-central cyclic subgroups contains exactly $p$ subgroups. Let $\left< d\right>$ be such a subgroup. By $\chi(d) \in Q$ we indicate that $\chi(d) = \zeta^q$ for some $q \in Q$. To compute the contribution of the remaining elements we distinguish two cases.
 
\emph{Case 1:} $(\varepsilon_g(d), \varepsilon_h(d)) = (1, 0)$, $\chi(d^{-1}) \in Q$ or $(\varepsilon_g(d), \varepsilon_h(d)) = (0, 1)$, $\chi(d^{-1}) \in N$.\\ Then the contribution of the conjugacy class of $\left<d \right>$ is \begin{align*}  p\sum_{h \in \langle d \rangle \setminus \{1\}} \eta(h)\chi(h^{-1})  =&\ p \left(\sum_{q \in Q}\left( \frac{p+\epsilon}{2} + p\sum_{i \in Q} \zeta^i \right) \zeta^q + \sum_{n \in N}\left( \frac{p+\epsilon}{2} + p\sum_{j \in N} \zeta^j \right) \zeta^n \right) \\ =&\ p\left( - \frac{p + \epsilon}{2} + p\left(\left( \sum_{q \in Q} \zeta^q \right) \left( \sum_{i \in Q} \zeta^i \right) + \left( \sum_{n \in N} \zeta^n \right) \left( \sum_{j \in N} \zeta^j \right) \right) \right) \\ =&\ p\left( -\frac{p+\epsilon}{2} + p\left( \frac{\epsilon \cdot p + 1}{2} \right) \right) \\ =&\ \epsilon \cdot \frac{p^3 - p}{2}  \end{align*}
 
\emph{Case 2:} $(\varepsilon_g(d), \varepsilon_h(d)) = (1, 0)$, $\chi(d^{-1}) \in N$ or $(\varepsilon_g(d), \varepsilon_h(d)) = (0, 1)$, $\chi(d^{-1}) \in Q$.\\ Then the contribution of the conjugacy class of $\left<d \right>$ is \begin{align*} &\ p \left(\sum_{n \in N}\left( \frac{p+\epsilon}{2} + p\sum_{i \in Q} \zeta^i \right) \zeta^n + \sum_{q \in Q}\left( \frac{p+\epsilon}{2} + p\sum_{j \in N} \zeta^j \right) \zeta^q \right) \\ =&\ p\left( - \frac{p + \epsilon}{2} + p\left(\left( \sum_{n \in N} \zeta^n \right) \left( \sum_{i \in Q} \zeta^i \right) + \left( \sum_{q \in Q} \zeta^q \right) \left( \sum_{j \in N} \zeta^j \right) \right) \right) \\ =&\ p\left( -\frac{p+\epsilon}{2} + p\left( \frac{-\epsilon \cdot p + 1}{2} \right) \right) \\ =&\ -\epsilon \cdot \frac{p^3 + p}{2} \end{align*}

From now on let $I = \{0, 1, ..., p-1 \}$. Let $\chi$ be a non-trivial linear character of $H$ and let $s \in \Ker (\chi) \setminus \Z(H)$. Moreover let $t \not\in \Ker(\chi)$ such that $\chi(t^{-1}) \in Q$. Then $\{(ts^i)^{-1} \mid i \in I \}$ is a set which contains exactly one element from every conjugacy class of cyclic subgroups not lying in $\Ker(\chi)$ and $\chi((ts^i)^{-1}) \in Q$ for every $i$. Set \[ \gamma = |\{\ i \mid (\varepsilon_g(ts^i), \varepsilon_h(ts^i)) = (1,0)\ \} | \] and \[ \delta = |\{\ i \mid (\varepsilon_g(ts^i), \varepsilon_h(ts^i)) = (0,1)\ \} |. \] Then $\gamma + \delta = p$ and, summing up all the contributions obtained above, we get \begin{align*} \langle \eta, \chi \rangle_H =&\ \frac{1}{p^3} \left( \frac{p^3 + \epsilon}{2} + \epsilon \cdot \frac{p-1}{2} + \epsilon p \cdot \frac{p-1}{2} + \epsilon \gamma \cdot \frac{p^3 - p}{2} - \epsilon\delta  \cdot \frac{p^3 + p}{2} \right) \\ =&\ \frac{1}{2p^3} \left( p^3 \left(1 + \epsilon \gamma - \epsilon \delta \right) + \epsilon p^2 - \epsilon(\gamma + \delta )p \right) \\ =&\  \frac{1 +\epsilon \gamma - \epsilon \delta}{2}.  \end{align*} For $n \in N$ the map $\chi^n \colon H \to \ZZ[\zeta] \colon x \mapsto \chi(x)^n$ is also a linear character of $H$ and an analogous computation gives \[ \langle \eta, \chi^n \rangle_H = \frac{1 -\epsilon \gamma + \epsilon \delta}{2}. \] Since both, $\langle \eta, \chi \rangle_H$ and $\langle \eta, \chi^n \rangle_H$, are non-negative, necessarily $|\gamma - \delta| \leq 1$. Thus $\gamma = \frac{p \pm 1}{2}$ and $\delta = \frac{p \mp 1}{2}$.

Recall that by \eqref{eq:defH}, $H$ is generated by a central element $z$ and two other elements $b$ and $c$. We may assume w.l.o.g.\ by \eqref{eq:power_n} that $(\varepsilon_g(c), \varepsilon_h(c)) = (1, 0)$. We have that $\{ \langle c \rangle, \langle bc^i \rangle \mid i \in I \}$ is a set of representatives of the $p + 1$ conjugacy classes of non-central cyclic subgroups of $H$. Up to the action of $\operatorname{Gal}(\QQ(\zeta)/\QQ)$, every linear character is determined by its kernel. Let $n \in N$. Then by \eqref{eq:power_n} there exist $a_0, a_1, ..., a_{p-1} \in \{1, n\}$ such that $c, b^{a_0}, b^{a_1}c^{a_1}, ..., b^{a_{p-1}}c^{(p-1)a_{p-1}}$ all have partial augmentation $1$ at $g$ and $0$ at $h$.

Let $\chi$ be a non-trivial linear character with $\langle c \rangle \subseteq \Ker(\chi)$ and $\chi(b^{-1})\in Q$. Then $S = \{bc^i \mid i \in I\}$ is a set which contains exactly one element from every conjugacy class of cyclic subgroups not lying in $\Ker(\chi)$. By the above, $\frac{p\pm1}{2}$ of the elements of $S$ have partial augmentation $1$ at $g$ and the other elements have augmentation $0$ at $g$. Since $\varepsilon_g(bc^i) = 1$ if and only if $a_i = 1$ we get \begin{align} |\{\ i \in I \mid a_i = 1\ \}| \in \left\{ \frac{p \pm 1}{2} \right\}. \label{eq:m} \end{align}

Now let $j \in I$ and $\chi$ be a non-trivial linear character of $H$ such that $\langle bc^j\rangle \subseteq \Ker(\chi)$ and $\chi(c^{-1}) \in Q$. For every $i \in I \setminus \{j\}$ we determine one element of the form $ b^{\ell a_i}c^{\ell i a_i}$, for some $\ell$, lying in the coset $\Ker(\chi)c = \left\{ z^r b^k c^{jk + 1} \mid 0 \leq k, r \leq p-1 \right\}$. To do so let $\ell, k$ be such that \[ b^k c^{jk + 1} =b^{\ell a_i} c^{\ell i a_i}. \] Thus $k \equiv \ell a_i \bmod p$ and $jk + 1 \equiv \ell i a_i \bmod p$. This gives $\ell \equiv (ia_i - ja_i)^{-1} \bmod p$. For the partial augmentations of $b^{\ell a_i} c^{\ell i a_i}$ it only matters, by \eqref{eq:power_n}, whether $\ell$ is a quadratic residue modulo $p$. Hence $b^{\ell a_i} c^{\ell i a_i}$ has the same partial augmentations as $c$ if and only if $\ell \in Q$, i.e.\ $ia_i - ja_i \in Q$. Hence \begin{align} 1+ |\{\ i \in I \setminus \{j\} \mid ia_i - ja_i \in Q\ \}| \in \left\{ \frac{p\pm 1}{2} \right\},\label{eq:mj} \end{align} where the $1$ represents the element $c$.

Denote by $(\ r \mid p\ )$ the Legendre symbol of $r$ modulo the prime $p$ and set $\beta_i = (\ a_i \mid p\ )$ for $i \in I$. By \eqref{eq:m} we have \begin{align} \sum_{i \in I} \beta_i =\sum_{i \in I} (\ a_i \mid p\ ) \in \left\{\pm 1 \right\}.\label{eq:msum} \end{align} From \eqref{eq:mj} we get for every $j \in I$: \[1 + \sum_{i \in I} (\ i-j \mid p\ )\beta_i = 1 + \sum_{i \in I} (\ (i-j)a_i \mid p\ ) \in \left\{\pm 1 \right\}. \] Set $s_i = (\ i \mid p\ )$ for $i \in \{1, ..., p-1\}$. Then the above equations read as follows: there exist $m_0, ..., m_{p-1} \in \{\pm 1\}$ such that \[ \left( \begin{array}{ccccc} \beta_0 & s_1 & s_2 & \dots & s_{p-1} \\ s_{p-1} & \beta_1 & s_1 & \dots & s_{p-2} \\ \vdots & \vdots & \ddots & \vdots & \vdots \\ s_2 & s_3 & \dots & \beta_{p-2} & s_1 \\ s_1 & s_2 & \dots & s_{p-1} & \beta_{p-1} \end{array} \right) \left( \begin{array}{c} \beta_0 \\ \beta_1  \\ \vdots \\  \beta_{p-2} \\ \beta_{p-1} \end{array} \right) = \left( \begin{array}{c} m_0 \\ m_1 \\ \vdots \\ m_{p-2} \\ m_{p-1} \end{array} \right). \] Summing up all equations and sorting with regard to the $\beta_i$'s we get \[ \beta_0(s_1 + ... + s_{p-1}) + ... + \beta_{p-1}(s_1 + ... + s_{p-1}) + p = m_0 + ... + m_{p-1}.\] Since $s_1 + ... + s_{p-1} = 0$ all the $m_i$'s must be equal to $1$. Thus we have to solve the homogeneous system of linear equations given by \[ \left( \begin{array}{ccccc} 0 & s_1 & s_2 & \dots & s_{p-1} \\ s_{p-1} & 0 & s_1 & \dots & s_{p-2} \\ \vdots & \vdots & \ddots & \vdots & \vdots \\ s_2 & s_3 & \dots & 0 & s_1 \\ s_1 & s_2 & \dots & s_{p-1} & 0 \end{array} \right) \left( \begin{array}{c} \beta_0 \\ \beta_1  \\ \vdots \\  \beta_{p-2} \\ \beta_{p-1} \end{array} \right) = \left( \begin{array}{c} 0 \\0 \\ \vdots \\0 \\ 0 \end{array} \right). \] The null space of the system has (at most) dimension $1$ by the following \hyperlink{lemma}{lemma} and contains the vector $(\beta_0, \beta_1, ..., \beta_{p-1})^t = (1, 1, ..., 1)^t$. This is however out of the question by \eqref{eq:msum}. \hfill $\Box$ 

\begin{lemma} \hypertarget{lemma}{} Let $m$ be odd and $A \in \FF_2^{m \times m}$ where \[ A =  \left( \begin{array}{ccccc} 0 & 1 & 1 & \dots & 1 \\1 & 0 & 1 & \dots & 1 \\ \vdots & \vdots & \ddots & \vdots & \vdots \\ 1 & 1 & \dots & 0 & 1 \\ 1 & 1 & \dots & 1 & 0 \end{array} \right).\] Then $\operatorname{rank} A = m-1$. \end{lemma}

\begin{proof} Summing up all rows of $A$ gives the zero vector, hence the rank of $A$ is at most $m - 1$. Let $z_1, ..., z_{m-1}$ be the first $m-1$ rows of $A$ and $\alpha_1, ..., \alpha_{m-1} \in \FF_2$ such that $\sum_{i = 1}^{m-1} \alpha_iz_i = 0$. Then we get $\sum_{i = 1, i\not= j}^{m-1} \alpha_i = 0$ for every $j \in \{1, ..., m-1\}$ and $\sum_{i = 1}^{m-1} \alpha_i = 0$. Hence $\alpha_j = 0$ for all $j \in  \{1, ..., m-1\}$. \end{proof}

\bibliographystyle{amsalpha}
\bibliography{non_abelian.bib}

\providecommand{\bysame}{\leavevmode\hbox to3em{\hrulefill}\thinspace}
\providecommand{\MR}{\relax\ifhmode\unskip\space\fi MR }
\providecommand{\MRhref}[2]{%
  \href{http://www.ams.org/mathscinet-getitem?mr=#1}{#2}
}
\providecommand{\href}[2]{#2}
\begin{thebibliography}{HHK09}

\bibitem[Ber55]{Ber}
S.D. Berman, \emph{On the equation {$x^m=1$} in an integral group ring},
  Ukrain. Mat. \v Z. \textbf{7} (1955), 253--261.

\bibitem[Ble95]{Bleher}
F.M. Bleher, \emph{Tensor products and a conjecture of {Z}assenhaus}, Arch.
  Math. (Basel) \textbf{64} (1995), no.~4, 289--298.

\bibitem[CL65]{CL}
J.A. Cohn and D.~Livingstone, \emph{On the structure of group algebras {I}},
  Canadian Journal of Mathematics \textbf{17} (1965), 583--593.

\bibitem[Dic58]{Dickson}
L.E. Dickson, \emph{Linear groups: {W}ith an exposition of the {G}alois field
  theory}, with an introduction by W. Magnus, Dover Publications, Inc., New
  York, 1958.

\bibitem[DJ96]{DoJu}
M.A. Dokuchaev and S.O. Juriaans, \emph{Finite subgroups in integral group
  rings}, Canad. J. Math. \textbf{48} (1996), no.~6, 1170--1179.

\bibitem[Dor71]{Dornhoff}
L.~Dornhoff, \emph{Group representation theory. {P}art {A}: {O}rdinary
  representation theory}, Marcel Dekker, Inc., New York, 1971, Pure and Applied
  Mathematics, 7.

\bibitem[Gau11]{Gauss}
C.F. Gau{\ss}, \emph{Summatio quarumdam serierum singularium}, Commentationes
  soc. reg. sc. Gotting. recentiores \textbf{1} (1811), Published in German in:
  Gauss, Untersuchungen {\"u}ber h{\"o}here Arithmethik, Chelsea Publishing
  Co., New York, 1965.

\bibitem[Her07]{HertweckBrauer}
M.~Hertweck, \emph{Partial {A}ugmentations and {B}rauer character values of
  torsion units in group rings},
  \href{http://arxiv.org/abs/math/0612429v2}{\nolinkurl{arXiv:math.RA/0612429v2
  [math.RA]}} (2007).

\bibitem[Her08]{CpCp}
\bysame, \emph{Unit groups of integral finite group rings with no noncyclic
  abelian finite {$p$}-subgroups}, Comm. Algebra \textbf{36} (2008), no.~9,
  3224--3229.

\bibitem[HHK09]{HHK}
M.~Hertweck, C.R. H{\"o}fert, and W.~Kimmerle, \emph{Finite groups of units and
  their composition factors in the integral group rings of the group {${\rm
  PSL}(2,q)$}}, J. Group Theory \textbf{12} (2009), no.~6, 873--882.

\bibitem[Hig40]{Hig}
G.~Higman, \emph{Units in group rings}, D. phil. thesis, Oxford Univ., 1940.

\bibitem[H{\"o}f08]{HoefPhD}
C.R. H{\"o}fert, \emph{Bestimmung von {K}ompositionsfaktoren endlicher
  {G}ruppen aus {B}urnsideringen und ganzzahligen {G}ruppenringen},
  Doktorarbeit, Universit{\"a}t Stuttgart, 2008,
  \url{http://elib.uni-stuttgart.de/opus/frontdoor.php?source_opus=3541&la=de}.

\bibitem[Hup67]{Huppert1}
B.~Huppert, \emph{{E}ndliche {G}ruppen {I}}, {D}ie {G}rundlehren der
  mathematischen {W}issenschaften ed., vol. 134, Springer-Verlag, Berlin, 1967.

\bibitem[Kim07]{C2C2}
W.~Kimmerle, \emph{Torsion units in integral group rings of finite insoluble
  groups}, Oberwolfach Reports \textbf{4} (2007), no.~4, 3229--3230, Abstracts
  from the mini-workshop held November 25--December 1, 2007, Organized by Eric
  Jespers, Zbigniew Marciniak, Gabriele Nebe, and Wolfgang Kimmerle.

\bibitem[LP89]{LutharPassiA5}
I.S. Luthar and I.B.S. Passi, \emph{Zassenhaus conjecture for {$A_5$}}, Proc.
  Indian Acad. Sci. Math. Sci. \textbf{99} (1989), no.~1, 1--5.

\bibitem[Mar14]{Mar_Sylow}
L.~Margolis, \emph{A {S}ylow theorem for the integral group ring of
  {PSL}(2,q)}, \href{http://arxiv.org/abs/1408.6075}{\nolinkurl{arXiv:1408.6075
  [math.RA]}} (2014).

\bibitem[Neu99]{Neukirch_english}
J.~Neukirch, \emph{Algebraic number theory}, Grundlehren der Mathematischen
  Wissenschaften [Fundamental Principles of Mathematical Sciences], vol. 322,
  Springer-Verlag, Berlin, 1999.

\bibitem[Sak71]{Saksonov}
A.~I. Saksonov, \emph{Group rings of finite groups. {I}}, Publ. Math. Debrecen
  \textbf{18} (1971), 187--209.

\bibitem[San81]{Sand1}
R.~Sandling, \emph{Graham {H}igman's thesis ``{U}nits in group rings''},
  Integral representations and applications ({O}berwolfach, 1980), Lecture
  Notes in Math., vol. 882, Springer, Berlin-New York, 1981, pp.~93--116.

\bibitem[San85]{Sand2}
\bysame, \emph{The isomorphism problem for group rings: a survey}, Orders and
  their applications ({O}berwolfach, 1984), Lecture Notes in Math., vol. 1142,
  Springer, Berlin, 1985, pp.~256--288.

\bibitem[Seh93]{SehgalBook2}
S.K. Sehgal, \emph{Units in integral group rings}, Pitman Monographs and
  Surveys in Pure and Applied Mathematics, vol.~69, Longman Scientific \&
  Technical, Harlow, 1993.

\end{thebibliography}
 
\end{document}